\documentclass[preprint,12pt]{elsarticle}
\usepackage{amssymb,latexsym,amsmath,enumitem,amscd,amsthm,graphicx,color}
\usepackage{graphicx}    
\usepackage{subcaption}  
\usepackage{caption}     
\usepackage{hyperref}
\usepackage{graphicx}
\usepackage{float} 
\usepackage{amsmath, amssymb, amsthm} 
\usepackage[all]{xy}
\usepackage{dsfont}
\usepackage{bbm}
\usepackage{enumerate}
\usepackage[utf8]{inputenc}
\usepackage{enumerate}
\usepackage{upgreek}
\usepackage{amsmath}
\usepackage{hyperref}
\usepackage{enumitem}
\usepackage{mathrsfs}
\usepackage{booktabs}    
\usepackage{multirow}    
\usepackage{graphicx,enumerate}    

\allowdisplaybreaks
\usepackage{easyReview} 
\newtheorem{lemma}{Lemma}[section]
\newtheorem{theorem}{Theorem}[section]

\newtheorem{algorithm}{Algorithm}[section]
\theoremstyle{remark}
\theoremstyle{definition}
\newtheorem{remark}{Remark}[section]

\newtheorem{example}{Example}[section]
\usepackage{tcolorbox}
\usepackage{scrlayer-scrpage}
\usepackage[T1]{fontenc}
\usepackage{booktabs}
\usepackage{caption}
\usepackage{multirow}

\usepackage{todonotes}

\usepackage{lipsum}
\journal{}

\begin{document}

\begin{frontmatter}

\title{A Conjugate Gradient Method for Nonlinear Programming Problems using Caputo Fractional Gradients}
\author{Barsha Shaw$^{a}$, Md Abu Talhamainuddin Ansary$^{a,}$\footnote{Corresponding author: md.abutalha2009@gmail.com}}
\address{$^{a}$Department of Mathematics, IIT Jodhpur, Jodhpur-342030, Rajasthan, India.}
\begin{abstract}
The article proposes a Caputo fractional conjugate gradient (CFCG) method for unconstrained optimization problems which is applicable to smooth as well as non-smooth problmes. The proposed method uses a non-adaptive version of the Caputo fractional derivative that provides integer-order derivatives information. A descent direction is obtained using the Caputo fractional gradients of two consecutive iterative points with a parameter ($\beta$). An inexact line search technique based on Armijo-Wolfe line conditions is used to find a suitable step length. Finally, a descent sequence is generated. The convergence results are derived under mild assumptions that ensuring of convergence is at least linear. Moreover, the convergence of the proposed method for quadratic functions is established through a Tikhonov-regularized formulation that can be interpreted as an extension of the least-squares approach. 
Finally, some numerical experiments, including neural network applications, are performed to justify that the proposed method achieves faster and more stable performance. 
\end{abstract}

\end{frontmatter}

\section{Introduction}
Optimization is a fundamental discipline with applications across nearly every field, where problems are typically formulated as the minimization or maximization of real-valued functions. To address such problems, one of the earliest and most intuitive techniques is the gradient descent method for non-linear optimization problems. This approach iteratively updates the solution by moving in the direction of the negative gradient, which represents the path of the fastest decrease in the function value. This is known as the steepest descent method.
However, steepest descent can suffer from slow convergence, especially in ill-conditioned problems. This drawback has led to the development of faster methods such as conjugate gradients (CG) (\cite{steepest1,Nocedal}). To achieve better convergence with lower computational cost, various forms of the conjugate gradient parameter $\beta $ have been developed. Notable examples include those proposed in \cite{FR,CD,DY,HS,PRP}. These methods generally rely on smooth objective functions. To handle non-smooth functions, proximal techniques have been introduced. The steepest proximal descent, the proximal CG, and the proximal Newton methods are discussed in \cite{proximal2,proximal_Newton}. These methods have many real-life applications, as shown in \cite{application1, application2, application3}.

\par An alternative approach for handling non-smooth functions involves the use of non-integer order derivatives, known as fractional derivatives. To capture both the non-smooth behaviour and the information associated with non-integer order differentiation, the concept of fractional calculus becomes useful. Furthermore, due to the integration term involved in the definition of a fractional derivative, it reflects long-memory effects and nonlocal characteristics, which cannot be represented by integer-order derivatives since they depend only on the local behaviour of functions. Additionally, due to the lack of memory effect, steepest descent, conjugate gradient etc method goes in a zigzag way. For this reason, fractional derivative played an advantageous role in research. In the theory of fractional derivatives, among the various existing definitions, the Caputo fractional derivative and the Riemann–Liouville (R–L) derivative are the most widely used. The study of fractional calculus has gained interest due to its applications across many fields such as mechanical and electrical engineering, signal processing, quantum mechanics, bioengineering, biomedicine, financial systems, machine learning, and neural networks (\cite{A1,A2,A3,A4}). The R-L fractional derivative  \( {}_aD^{\alpha}_x \) and the Caputo fractional derivative  \( {}_a^cD^{\alpha}_x \) are two standard formulations is present in \cite{20}. In this article, monotonicity results for both definitions are presented on the interval \( (0,1) \) which are later extended to generalized intervals \( (n, n+1) \) for \( n \in \mathbb{N} \) in \cite{Barsha1}. Several theoretical results for fractional derivatives can also be found in \cite{23}. A necessary optimality condition by using the Caputo derivative is developed in \cite{21}. The Taylor–Riemann series and mean value theorem for multivariate functions involving the Caputo derivative are introduced in \cite{MVT,27}. A gradient descent method based on Caputo derivatives is also proposed there, which relies on a fractional Taylor expansion. Further motivation for employing the Caputo fractional derivative in steepest descent methods can be found in \cite{Barsha2}. However, to reduce high computational complexity and the number of iterations, different strategies can be adopted. For this reason, the present article focuses on a Caputo fractional conjugate gradient method. Main contributions of this article is as follows:
\begin{itemize}
    \item A conjugate gradient method is developed for nonlinear optimization using an fractional order Caputo fractional gradient.
    \item An algorithm is proposed with its convergence results which justified under mild assumptions (Theorems \ref{convergence thm}-\ref{conv_trik}).
    \item The proposed method is compared with existing methods using numerical examples based on neural networks. We have adopted some ideas from \cite{neural1,neural2,NN,NN1} in this regard.
\end{itemize}
\par The structure of the article is as follows.
Section~\ref{sec1} provides the basic definitions and preliminaries required for present developments. In Section~\ref{sec2}, the CFCG method is introduced with different forms of \( \beta \). Section~\ref{sec3} proposed an algorithm for the convergence analysis, specifically for Tikhonov regularisation solutions. Section~\ref{sec4} gives numerical results to show the performance of the method, including examples of neural networks. Section~\ref{sec5} concludes this article with some suggestions on possible directions for future research.
\section{Preliminaries}\label{sec1}
We consider the unconstrained optimization problem
\[
{(P)}: \quad \min_{x \in \mathbb{R}^n} f(x)
\]
where $f:\mathbb{R}^n\rightarrow\mathbb{R}$ is a continuous function. The steepest descent method is considered a classical numerical technique for solving $(P)$ type problem. In this method, every iterating point \( x^k \) moves along the descent direction \(d^k = -\nabla f(x^k) \). However, the linear rate of convergence is a major limitation of this method. This method was further modified in different directions to achieve accelerated convergence. The CG method is one such modification. In this method, the following search direction at $x^k$ is
\[
d^{k} =\begin{cases}
 -\nabla f(x^{0}),\quad \quad \quad \quad~\qquad\space k=0\\
 -\nabla f(x^{k}) + \beta_{k}\, d^{k-1}, \qquad k \ge 1,

\end{cases} 
\]
In literature, different choices \( \beta ~(\text{for}~k\geq 1)\) are used for the CG method to improve convergence speed. In this article, we have considered the following different types of $\beta$
\begin{align*}
    &\text{Fletcher-Reever}~ (\cite{FR}): \beta_k^{FR} = \frac{\langle \nabla f(x^k), \nabla f(x^k) \rangle}{\langle \nabla f(x^{k-1}), \nabla f(x^{k-1}) \rangle} \\
    &\text{Conjugate-Descent (\cite{CD})}: \beta_k^{CD} = -\frac{\langle \nabla f(x^k), \nabla f(x^k) \rangle}{\langle d^{k-1}, \nabla f(x^{k-1}) \rangle} \\
    &\text{Dai-Yuan ($DY$) (\cite{DY})}: \beta_k^{DY} = \frac{\langle \nabla f(x^k), \nabla f(x^k) \rangle}{\langle d^{k-1}, \nabla f(x^k) - \nabla f(x^{k-1}) \rangle} \\
    &\text{Polak-Ribiere-Polyak (\cite{PRP})}: \beta_k^{PRP} = \frac{\langle \nabla f(x^k), \nabla f(x^k) - \nabla f(x^{k-1}) \rangle}{\langle \nabla f(x^{k-1}), \nabla f(x^{k-1}) \rangle} \\
    &\text{Hestenes-Stiefel (\cite{HS})}: \beta_k^{HS} = \frac{\langle \nabla f(x^k), \nabla f(x^k) - \nabla f(x^{k-1}) \rangle}{\langle d^{k-1}, \nabla f(x^k) - \nabla f(x^{k-1}) \rangle}.
\end{align*}
Throughout this article, the convergence analysis of the proposed CFCG method requires the search direction satisfy sufficient descent condition
$
\nabla f(x^k)^\top d^k < -r\,\|d^k\|^2,~r > 0.$

Next, step size \( \eta_k \) is determined through using the following Armijo-Wolfe conditions,
\begin{align*}
    & f(x^k + \eta_k~ d^k) \leq f(x^k) + c_1 \eta_k~ \nabla f(x^k)^\top d^k, \\
    & \nabla f(x^k + \eta_k~ d^k)^\top d^k \geq c_2 \nabla f(x^k)^\top d^k,
\end{align*}
where \( 0 < c_1 < c_2 < 1 \). Finally, the next iterated point is calculated as $x^{k+1} = x^{k} + \eta_k d^{k}$.

 \par To accommodate non-smooth behaviour, we work with the class \( A^n[a,b] \) of functions that are \( n \)-times differentiable and whose \( n+1 \)-th derivative is absolutely continuous on \([a,b]\). For a function \( f \in A^n[a,b] \), the Caputo fractional derivative of order \( \alpha \in (n, n+1) \) at a point \( x \in [a,b] \) is defined as 

  \begin{equation*}
      {}^C_cD^{\alpha}_xf(x) = \frac{1}{\Gamma(n-\alpha)}  \int_{c}^{x}(x-\tau) ^{n-\alpha-1} f^{(n)}(\tau)d\tau,
       \end{equation*}
 A multivariate continuous function, the fractional gradient denoted by $^C_c\bar{\nabla}^{\alpha}_xf(x)$ \cite{38}, is defined as for $c\in\mathbb{R}^n$, 
 \begin{equation*}
     ^C_c\bar{\nabla}^{\alpha}_xf(x) =\Big(~
         ^C_{c_1}D^{\alpha}_{x_1}f(x), ~
         ^C_{c_2}D^{\alpha}_{x_2}f(x), \dots,~ ^C_{c_n}D^{\alpha}_{x_n}f(x)\Big)^\top.
 \end{equation*} 
 Although based on multi-variate fractional gradient, a mean value theorem for the  Caputo fractional derivative is available (Theorem~15 in \cite{MVT}) from Taylor series. But this formulation is unable to give an optimal solution. Throughout this article, we use a modified Caputo derivative obtained from the classical Taylor series expansion. Therefore, Caputo derivative is used in the place of classical derivative form of the mean value theorem.
\begin{lemma}\label{thm1}(Theorem 3.2, \cite{27})
Let $ f_{\alpha,\rho}: \mathbb{R}^n\rightarrow \mathbb{R} $ admit the following Taylor series expansion around $c\in\mathbb{R}^n$ for some $ \alpha \in (0,1) $ and $ \rho \in \mathbb{R} $
\begin{equation}\label{Taylor}
{}f_{\alpha,\rho}(x) = f(c) + \nabla f(c)^\top (x - c) + \frac{1}{2} C_{\alpha,\rho}~ (x-c)^\top H(\xi)(x - c)
\end{equation}
where $C_{\alpha,\rho} = \frac{\Gamma(2-\alpha)\Gamma(2)}{\Gamma(3-\alpha)} + \rho $, $H$ is  the Hessian of function $f_{\alpha,\rho}$ based on integer-order derivative where $c<\xi<x$. The integer order gradient coincides with the Caputo gradient at $c$. We define the Caputo fractional gradient as
\[
{}^C_c\nabla^{\alpha}_x~{}f(x) 
= diag\Big({}^C_c\nabla^{\alpha}_x I(x)\Big)^{-1}
\Big( {}^C_c\nabla^{\alpha}_x f_{\alpha,\rho}(x) + \rho\, diag(|x - c|)\cdot {}^C_c\nabla^{1+\alpha}_x f_{\alpha,\rho}(x)\Big),
\]
where $ I : \mathbb{R}^n \to \mathbb{R} $ is the identity map.
\end{lemma}

\section{ Caputo fractional conjugate gradient method.}\label{sec2}
The primary focus of this section is to implement the Caputo fractional derivative in the context of CG method. The motivation for using the Caputo derivative instead of the classical derivative is briefly discussed in \cite{Barsha2}. We have used $\alpha=0.9$ and $\rho =0.1$ to find fractional gradient at any iteration point $x^k$. A change in the value of $\alpha$ yields valuable insights into the behaviour of the non-integer order derivative. These findings highlight the enhanced performance of the proposed CFCG method. In addition to this, the theoretical framework of the CFCG method is developed in this section. 
\par Following the arguments of CG methods, a search direction $d^k$ of $f$ at point $x^k$ is defined as  

\begin{equation}\label{eq:d}
d^{k} =
\begin{cases}
-\,{}^{C}_{c}\nabla_{x}^{\alpha} f(x^{k}), & k=0,\\[4pt]
-\,{}^{C}_{c}\nabla_{x}^{\alpha} f(x^{k}) + \beta_{k}\, d^{k-1}, & k\ge 1,
\end{cases}
\end{equation}
where $\beta_{k}$ is an algorithmic parameter. For simplicity ${}^{C}_{c}\nabla_{x}^{\alpha} f(x^k)$ is written  as \(g^k\), throughout the article. 
A sequence $\{x^k\}$ for $k>0$ is generated by,
\begin{equation*}
x^{k+1} = x^{k} + \eta_k d^{k}.
\end{equation*}
Different choices of \(\beta_{k}\) lead to different variants of the method. Since we are using the search directions of two consecutive iterations under the proper selection of step size ($\eta_k$) is necessary to guarantee that $d^k$ is a descent direction of $f$ at $x^k$ for every $k\geq1$. In our derivations, we have used the following fractional Armijo-Wolfe inexact line search techniques using Caputo fractional gradient
\begin{align}
      & f(x^k+ \eta_k d^k)\leq f(x^k) + c_1 \eta_k ~g^k{^\top} d^k\label{eq9}\\
       & {}^{C}_{c}\nabla_{x}^{\alpha} f(x^k+ \eta_k d^k){^\top} d^k\geq c_2 ~g^k{^\top} d^k.\label{eq8}
   \end{align}
with $0 < c_1< c_2 < 1$ prespecified scalars. Inequality \eqref{eq9} guarantees a sufficient decrease in objective function that step sizes are not too large while \eqref{eq8} ensures that they are not too small within the specified interval.

 The following lemma shows that the step size procedures are well defined.
\begin{lemma}\label{lemma2}
 Suppose $f:\mathbb{R}^n \to \mathbb{R}$ is an absolutely continuous and bounded below function and $d^k$ be a descent direction to $f$ at $x^k$. Then there exists some ${\eta_k}>0$ satisfying \eqref{eq9} and \eqref{eq8}. 
\end{lemma}
\begin{proof}
    Since $f$ is bounded below, $\phi(\eta;(x^k,d^k))$ defined by $\phi(\eta;(x^k,d^k)):= f(x^k + \eta d^k)$ is bounded below for all $\eta>0.$ Since $0<c_1<c_2<1,$ and $d^k$ is descent direction of $f$ at point $x^k$ then let construct an auxiliary function $$l(\eta_k;(x^k,d^k))= f(x^k)+ c_1~\eta_k~~g^k{^\top}~d^k$$ is unbounded below for all $\eta_k>0$. Then it must intersect the graph $\phi$ at least once. 
    Let $\eta_k'>0$ be the smallest intersecting value of $\eta_k$, and thus
    \begin{equation}\label{eq10}
        f(x^k+\eta_k' d^k)= f(x^k)+ c_1~\eta_k'~g^k{^\top}~d^k.
    \end{equation}
    The sufficient decrease condition \eqref{eq9} holds for all step size less than $\eta_k'$. Thus choose $\eta_k\in(0,\eta_k')$ such that \eqref{eq9} holds that is
    \begin{equation*}
        f(x^k+\eta_k d^k)\leq f(x^k)+ c_1~\eta_k~g^k{^\top}~d^k.
    \end{equation*}
    By mean-value theorem there exists $\eta_k''\in(0,\eta'_k)$ such that 
    \begin{equation}\label{eq11}
        f(x^k+\eta_k' d^k)-f(x^k)= \eta_k' ~^C_{c} \nabla_{x}^{\alpha} f(x^k+\eta_k'' d^k)^\top~ d^k.
    \end{equation}
    By combining \eqref{eq10} and \eqref{eq11}, we have
    \begin{equation*}
    ~^C_{c} \nabla_{x}^{\alpha} f(x^k+\eta_k'' d^k)^\top~ d^k =  c_1~g^k{^\top}~d^k  > c_2~ g^k{^\top}~d^k. 
    \end{equation*}
Since $0<c_1<c_2<1$ and $g^k{^\top}~d^k<0$ and $f$ is continuous, there is an interval around $\eta_k''$ for which Wolfe conditions \eqref{eq8} hold.
\end{proof}
Next we need to show that $d^k$ ensures sufficient decease for all $k$. Since $d^0$ is same as the descent direction of Caputo fractinal gradient based descent method in \cite{27}, $d^0$ provides sufficient decrease. For different choice of $\beta$, we have proved a sufficient descent direction $d^k$ for $r>0$ such that 

 \begin{equation}\label{sufficient_dD}
     g^k{^\top} d^k \leq -r\,\|g^k\|^2
 \end{equation}
 holds for all $k$, in the following lemmas (Lemmas \ref{theorem1}-\ref{theorem3}).

\begin{lemma}\label{theorem1}(For $FR$ type $\beta$)
Let \( f:\mathbb{R}^n \to \mathbb{R} \) be an absolute function.  
Suppose that the sequence \( \{x^k\} \) is generated by the CG method, and the search direction $d^k$ is computed by $d^k = -g^k + \beta_k d^{k-1}, $ where  $\beta_k^{FR} = \frac{\|g^k\|^2}{\|g^{k-1}\|^2}$.
Then \( d^k \) satisfies a sufficient descent direction \eqref{sufficient_dD}. 
\end{lemma}

\begin{proof}
We have
\[
g^k{}^\top d^k = g^k{}^\top(-g^k + \beta_k d^{k-1}) 
= -\|g^k\|^2 + \beta_k^{FR}\, g^k{}^\top d^{k-1}
\]
substituted the expression of \( \beta_k^{FR} \), we obtain
\begin{align*}
g^k{}^\top d^k 
= &-\|g^k\|^2 + \frac{\|g^k\|^2}{\|g^{k-1}\|^2}\, g^k{}^\top d^{k-1} \\
= &-\|g^k\|^2 \left( 1 - \frac{g^k{}^\top d^{k-1}}{\|g^{k-1}\|^2} \right).
\end{align*}
From the Wolfe condition for \(0 < c_2 < 1\) and \(g^{k-1}{}^\top d^{k-1}\leq - r_1\|g^{k-1}\|^2, r_1>0
\) it follows that
\begin{align*}
g^k{}^\top d^k 
&\leq -\|g^k\|^2 \left( 1 - c_2  \frac{g^{k-1}{}^\top d^{k-1}}{\|g^{k-1}\|^2} \right)\\
&\leq -(1+ c_2 r_1)\|g^k\|^2,
\end{align*}
 it is clear that \( g^k{}^\top d^k <  -r\|g^{k}\|^2\) for $r= 1+c_2~r_1>0$.  
Thus, \( d^k \) satisfies sufficient descent direction for $FR$-type $\beta$.  
\end{proof}

\begin{lemma}\label{theorem2}(For $CD$ and $DY$ type $\beta$)
Let \( f:\mathbb{R}^n \to \mathbb{R} \) be an absolute function.  
Suppose that the sequence \( \{x^k\} \) is generated by the CG method, and the search direction $d^k$ is computed by $d^k = -g^k + \beta_k d^{k-1}, $ where 
\begin{itemize}
\item[1.] $\beta_k^{CD} = -\frac{g^k{}^\top g^k}{g^{k-1}{}^\top d^{k-1}}$
\item[2.] $\beta_k^{DY} = \frac{\|g^k\|^2}{d^{k-1}{}^\top (g^k - g^{k-1})},$
\end{itemize}
then \( d^k \) satisfies the sufficient descent direction condition \eqref{sufficient_dD}. 
\end{lemma}

\begin{proof}
\begin{itemize}
    \item[1.] We compute
\[
g^k{}^\top d^k  = -\|g^k\|^2 + \beta_k^{CD}\, g^k{}^\top d^{k-1}.
\]
Substituting the expression for \( \beta_k^{CD} \), we obtain
\begin{align*}
g^k{}^\top d^k 
=& -\|g^k\|^2 - \frac{\|g^k\|^2}{g^{k-1}{}^\top d^{k-1}}\, g^k{}^\top d^{k-1} \\
=& -\|g^k\|^2 \left( 1 + \frac{g^k{}^\top d^{k-1}}{g^{k-1}{}^\top d^{k-1}} \right).
\end{align*}
To ensure the proof, it suffices to show that
\begin{equation*}\label{eq-CD}
1 + \frac{g^k{}^\top d^{k-1}}{g^{k-1}{}^\top d^{k-1}} > 0.
\end{equation*}
Using the Wolfe condition for $0 < c_2 < 1$, we have
$g^k{}^\top d^k \leq -r~\|g^k\|^2 $ where $r = 1+c_2>0$. Hence, \( d^k \) satisfies sufficient descent direction for the $CD$ type $\beta$.
\item[2.] Substituting the $DY$ formula for \( \beta \), we get
\begin{align*}
g^k{}^\top d^k =& -\|g^k\|^2 + \frac{\|g^k\|^2}{d^{k-1}{}^\top (g^k - g^{k-1})} g^k{}^\top d^{k-1}\\
=& -\|g^k\|^2 \left( 1-\frac{g^k{}^\top d^{k-1}}{d^{k-1}{}^\top (g^k - g^{k-1})}  \right).
\end{align*}
To ensure the sufficient descent condition, it is sufficient to show that
\[
1-\frac{g^k{}^\top d^{k-1}}{d^{k-1}{}^\top (g^k - g^{k-1})} >0.
\]
Now, using Wolfe conditions for $0<c_2<1$
\begin{align*}
  &1-\frac{g^k{}^\top d^{k-1}}{d^{k-1}{}^\top (g^k - g^{k-1})}\\
  &= 1- \frac{g^k{}^\top d^{k-1}}{g^k{^\top} d^{k-1} - g^{{k-1}^\top} d^{k-1}}\\
  &= 1 -\frac{c_2}{c_2-1}.
\end{align*}
 Thus, $g^k{}^\top d^k\leq -r \|g^k\|^2$ holds for $r= 1 -\frac{c_2}{c_2-1}>0$. So, $d^k$ satisfies a sufficient descent direction for $DY$ type $\beta$.
\end{itemize}
\end{proof}
Since the CFCG methods based on \( \beta^{\mathrm{PRP}} \) and \( \beta^{\mathrm{HS}} \) may cycle without converging to a solution, convergence can be ensured by using the modified formula\\ $\beta_k = \max\{0, \beta_k^{\mathrm{PRP}}\} \quad \text{and} \quad \beta_k = \max\{0, \beta_k^{\mathrm{HS}}\}$. To prove that the directions $d^k$ are sufficient descent, we first need the Lipschitz continuity about function $g$.

 We assume that the fractional gradient \( g(x) \) is Lipschitz continuous. That is, there exists a constant \( L > 0 \) such that for all \( x, y \in \mathbb{R}^n \), the following inequality holds \cite{Lipschitz},
\begin{equation}\label{Lipschitz}
\|g(x) - g(y)\| \leq L \|x - y\|.
\end{equation}
   
\begin{lemma}\label{theorem3}(For $PRP$ and $HS$ type $\beta$)
Let \( f : \mathbb{R}^n \to \mathbb{R} \) be a function that is absolutely continuous and satisfies the Lipschitz continuity condition. Suppose that the sequence \(\{x^k\}\) is generated by the nonlinear CG method with search direction is computed as $d^k = -g^k + \beta_k d^{k-1},$
where
\begin{itemize} 
   \item[1.] $\beta_k^{\text{PRP}} = \frac{g^k{}^\top (g^k - g^{k-1})}{\|g^{k-1}\|^2}$
 \item[2.] $\beta_k^{\text{HS}} = \frac{g^k{}^\top (g^k - g^{k-1})}{d^{k-1}{}^\top (g^k - g^{k-1})}$,
\end{itemize}
then \( d^k \) satisfies a sufficient descent direction \eqref{sufficient_dD}. 
\end{lemma}
\begin{proof}
From Wolfe condition, We have
\begin{align*}
    |g^k{}^\top(g^k - g^{k-1})| &\leq (1-c_2) \|g(x^k)\|^2.
\end{align*}
Again from Lipschitz condition for constant $M>0$, we have by inductive way
\begin{align*}
    \|g^{k-1}\|&\leq \|g^k\| +\|g^{k-1} -g^k\|\\
    & \leq\|g^k\| + L \eta_{k-1} \|d^{k-1}\|\\
    & \leq\|g^k\| + L \eta_{k-1} m\|g^{k-1}\|\\
    &\leq M\|g^k\|.
\end{align*}
\begin{itemize}
    \item[1.] For $PRP$ type $\beta$, we have $\beta_k^{PRP}= \max\{0,\beta_k^{PRP}\}$
    \begin{align*}
        |\beta_k^{PRP}|\leq \frac{g^k{}^\top (g^k - g^{k-1})}{\|g^{k-1}\|^2}\leq (1-c_2)\frac{\|g^k\|^2}{\|g^{k-1}\|^2}\leq \frac{1-c_2}{M^2}\leq 1.
    \end{align*}
Now, for $0\leq \beta_k^{PRP}\leq 1$
\begin{align*}
    g^k{}^\top d^k &\leq -\|g^k\|^2 +  g^k{}^\top d^{k-1}\\
    &\leq -\|g^k\|^2 + c_2\|g^{k-1}\|^2 \\
    &\leq -(1+ c_2M^2)\|g^k\|^2.
\end{align*}
Thus, $d^k$ satisfies sufficient descent direction, $g^k{}^\top d^k\leq -r\|g^k\|^2$ for\\ $r=(1+ c_2 M^2)>0$.
\item[2.] Now consider the $HS$ method for $0<c_2<1$
\begin{align*}
    &d^{k-1}{}^\top (g^k - g^{k-1})\geq (c_2-1) g^{k-1}{}^\top d^{k-1}.
\end{align*} 
Since, $g^{k-1}{}^\top d^{k-1}<0$ then $d^{k-1}{}^\top (g^k - g^{k-1})>0$ thus $\beta_k^{HS}$ is well defined for Caputo fractional derivative.
\begin{align*}
      g^k{}^\top d^k &\leq -\|g^k\|^2 + \beta_k^{HS} g^k{}^\top d^{k-1}\\
    &\leq -\|g^k\|^2 + \frac{(1-c_2)\|g^k\|^2}{(c_2-1)g^{{k-1}^\top} d^{k-1}} c_2 g^{{k-1}^\top} d^{k-1}\\
    &\leq -(1+ c_2) \|g^k\|^2.
\end{align*}
So $g^k{}^\top d^k\leq - r\|g^k\|^2$ hold for $r= 1+ c_2.$ Thus $\beta_k^{HS}$ satisfies sufficient descent direction.
\end{itemize}
\end{proof}
So far, it has been established that the directions \(d^{k}\) defined in \eqref{eq:d} are sufficient descent directions for \(f\) at \(x^{k}\) and that the associated step size satisfies the Armijo–Wolfe conditions. Similarly, for each sufficient descent direction \(d^{k}\), Lemma~\ref{lemma2} ensures the existence of a step size \(\eta_k>0\). Next section contains proposed algorithm  convergence results.

\section{Algorithm and convergence results.}\label{sec3}
This section presents the proposed algorithm along with its convergence analysis. In the following table, we propose the CFCG algorithm for problem $(P)$.\\
 \noindent\rule{\linewidth}{0.4pt} \begin{algorithm}\label{algorithm 1} \textbf{(CFCG method):}
\begin{itemize}
    \item[Step 1.]  {\it Initialization:} Supply $f$, initial point $x^0 \in \mathbb{R}^n$, scalars $\sigma,r \in (0, 1), c_1,c_2\\(0<c_1<c_2<1)$,  and tolerance $\epsilon > 0$. Set  $k = 0$.
    \item[Step 2.]  {\it Optimality check:} Compute $g^k$. If $\|g^k\| < \epsilon$, then stop and declare $x^k$ as an approximate stationary point. Otherwise, go to the next step.
    \item[Step 3.] {\it Descent direction:}  compute descent direction $d^k$ using (\ref{eq:d}).    \label{ste2} 
     \item[Step 4.] {\it Line search:} Select $\eta_k$ as the first element in the sequence $\{1, r, r^2, \ldots\}$ that satisfies Armijo-Wolfe conditions \eqref{eq8}.\label{line_search}
       \item[Step 5.] {\it Update:} Compute $x^{k+1} = x^k + \eta_k d^k$. Update $k \gets k + 1$ and go to Step \hypertarget{step2}{2}.
\end{itemize}
\end{algorithm}
\noindent\rule{\linewidth}{0.4pt}  
Now, this information is ready to present convergence analysis under some assumptions for the CFCG method. For all five choices of $\beta$, the denominator norm is positive. Therefore, the convergence proof is identical. The statement of convergence analysis is stated below.

\begin{theorem}\label{convergence thm}
Uppose $f:\mathbb{R}^n\to\mathbb{R}$ be absolutely continuous on an open set $\mathcal{N}$ containing the level set
$
\mathcal{L}:=\{x\in\mathbb{R}^n:\ f(x)\le f(x^0)\},$ where $x^0$ is the initial point. Further suppose that $\mathcal{L}$ is bounded, $g$ is Lipschitz continuous on $\mathcal{N}$, and $\cos^2\theta_k \ge \delta>0$ holds for all $k$ where $\theta_k$ is the angle between $d^k$ and $g^k$. Then any accumulation point of $\{x^k\}$, generated by Algorithm~\ref{algorithm 1}, is a stationary point of $(P)$.
\end{theorem}

\begin{proof} From Wolfe condition \eqref{eq8}, we have
\begin{align}\label{one;}
    (c_2-1)g^{k^\top}d^k\leq (g^{k+1}- g^k)^\top d^k.
    \end{align}
From \eqref{Lipschitz} for there exists $L>0$ such that 
\begin{align}\label{two}
\big(g^{k+1}- g^k\big)^\top d^k\leq L~ \eta_k\|d^k\|^2.
\end{align}
Combining \eqref{one;} and \eqref{two}, we obtain
\begin{equation*}
    (c_2 -1) g^k{^\top}d^k\leq L~ \eta_k\|d^k\|^2
\end{equation*}
Since $\|d^k\|\neq 0$ and ${g^k}^Td^k<0$
\begin{align*}
   \eta_k {g^k}^Td^k \leq\frac{c_2-1}{L} \frac{(g^k{^{\top}} d^k)^2}{\|d^k\|^2}.
\end{align*}
Then by Armijo condition \eqref{eq9},
\begin{align}
f(x^{k+1})\leq f(x^k) + \frac{c_1(c_2-1)}{L} \frac{(g^k{^\top} d^k)^2}{\|d^k\|^2}.\label{amj2}
\end{align}
This implies
\begin{align}\label{three}
f(x^0)-f(x^{k+1}) \geq \frac{c_1(1-c_2)}{L}\sum_{i=0}^{k} \frac{\big(g^i{}^{\top} d^i\big)^2}{\|d^i\|^2}=\frac{c_1(1-c_2)}{L}\sum_{i=0}^{k}\|g^i\|^2\cos^2\theta_i.
\end{align}
Since $x^k\in \mathcal{L}$ for all $k$, $\{x^k\}$ is a bounded sequence. Hence $\{f(x^k)\}$ is bounded. Then there exists $f^*$ such that $f(x^k)\geq f^*>-\infty$ for all $k$. Considering limit $k\rightarrow \infty$ in (\ref{three}), we have
$$ \frac{c_1(1-c_2)}{L}\sum_{i=0}^{\infty}\|g^i\|^2\cos^2\theta_i \leq f(x^0)-f^*<\infty.$$
This implies $\underset{k\rightarrow \infty}{\lim}~\|g^k\|^2\cos^2\theta_k=0 $. Since $\cos\theta_k\ge \delta>0$ holds for all $k$,  $\underset{k\rightarrow \infty}{\lim}~\|g^k\|=0 $.
Since $\{x^k\}$ is bounded, $\{x^k\}$ has at least one convergence sequence. Suppose $\{x^k\}_{k\in K_1}$ be any arbitrary convergent subsequence converging to $x^*$. Then $\{g^k\}_{k\in K_1}$ is a converging subsequence of $\{g^k\}$ which converges to $g(x^*).$ Since $\underset{k\rightarrow \infty}{\lim}~g^k=0$, $g(x^*)=0$. This implies $x^*$ is a stationary point of $(P)$. Since  $\{x^k\}_{k\in K_1}$ be any arbitrary convergent subsequence of $\{x^k\}$, any accumulation point of $\{x^k\}$ is a stationary point of $(P)$.
\end{proof}
The next theorem confirms the rate of convergence for strongly convex functions. 

\begin{theorem}\label{theorem8}
Suppose assumptions of Theorem \ref{convergence thm} are hold, and $\{x^{k}\}$ is a sequence generated by the algorithm (\ref{algorithm 1}) that converges to $x^*$. Further suppose there exists $0<m\le M$ such that
 \begin{equation}\label{one}
 m\|y\|^2 \le y^\top H(x)\,y \le M\|y\|^2
 \end{equation}
 holds for every $x,y\in\mathbb{R}^n$. Then $\{x^k\}$ converges at least linearly.
\end{theorem}
\begin{proof}
 Since $(g^k)^{\top} d^k<0$ and $\cos^2\theta_k\geq \delta$ hold for every $k$,
\[
\frac{-\, (g^k)^{\top} d^k}{\|g^k\|\,\|d^k\|}=|\cos\theta_k| \geq \sqrt{\delta}.
\]

From \eqref{amj2},
\begin{equation}\label{eq32}
f(x^k)-f(x^{k+1})
\;\ge\;
\rho_1\,\frac{\big((g^k)^{\top} d^k\big)^2}{\|d^k\|^2}=\rho_1\|g^k\|^2\left(\frac{-\, (g^k)^{\top} d^k}{\|g^k\|\,\|d^k\|}\right)^2 \geq \rho_1 \delta \|g^k\|^2.
\end{equation}
where $\rho_1=\frac{c_1(1-c_2)}{L}>0$. Since $x^*$ is a stationary point, we have $g(x^*)=0$. Using \eqref{Taylor} and \eqref{one}, we obtain

\begin{equation}\label{eq31}
\frac{m}{2}C_{\alpha,\rho}\,\|x-x^*\|^2 \le f(x)-f(x^*) \le \frac{M}{2}C_{\alpha,\rho}\,\|x-x^*\|^2 .
\end{equation}

By the mean value theorem,
\[
g(x)-g(x^*)=H(\xi)\,(x-x^*)
\]
for some $\xi$ on the line segment joining $x$ and $x^*$. Since $g(x^*)=0$, 
\[
g(x)=H(\xi)\,(x-x^*).
\]
Hence, using the bounds $mI \leq H(\xi)\leq MI$, we obtain
\begin{equation}\label{eq30}
m\,\|x-x^*\|\le \|g(x)\|\le M\,\|x-x^*\|.
\end{equation}
Now using \eqref{eq30} in \eqref{eq32} we have,
\[
f(x^k)-f(x^{k+1})
\;\ge\;
\delta~\rho_1\,m^2\,\|x^k-x^*\|^2.
\]
Finally from \eqref{eq31}, 
\begin{equation}\label{equation}
f(x^k)-f(x^{k+1})
\;\ge\;
\frac{2\delta\rho_1 m^2}{M~C_{\alpha,\rho}}\,\big(f(x^k)-f(x^*)\big).
\end{equation}

Let
$
\bar\rho:=\frac{2\delta\rho_1 m^2}{M~C_{\alpha,\rho}}.$
Then \eqref{equation} implies
\begin{eqnarray*}
f(x^{k+1})-f(x^*)&\le& (1-\bar\rho)\,\big(f(x^k)-f(x^*)\big),\\
& \le&(1-\bar\rho)^2\,\big(f(x^{k-1})-f(x^*)\big) \\
& \le& \dots\\
&\le& (1-\bar\rho)^k\,(f(x^0)-f(x^*))
\end{eqnarray*}
Using \eqref{eq31}, we obtain
\[
\|x^k-x^*\|^2 \le \frac{2}{m}\,\big(f(x^k)-f(x^*)\big)
\le \frac{2}{m}\,(1-\bar\rho)^k\,(f(x^0)-f(x^*)).
\]
Hence,
\[
\|x^k-x^*\|
\le (1-\bar\rho)^{k/2}\sqrt{\frac{2\,(f(x^0)-f(x^*))}{m}},
\]
which shows that $\{x^k\}$ converges to $x^*$ at least linearly.
\end{proof}

\textbf{Convergence analysis for the Tikhonov regularized least squares problem:}
Next focus is to illustrate the convergence property of the proposed fractional algorithm, considering the quadratic objective problem is of the form,
\begin{equation}\label{eq20}
\min_x\quad f(x)= \frac{1}{2}x^\top~ A x+ b^\top~ x + c,
\end{equation}
where $b,x\in\mathbb{R}^n, A =(a_{i,j})\in\mathbb{R}^{n\times n}$ and $c\in\mathbb{R}$. Here, it assumes the matrix $A$ is a positive symmetric definite matrix. The unique minimizer is $x^*=-A^{-1}b.$
The least square problem of \eqref{eq20} is written as 
\begin{equation}\label{eq21}
    \min_x\quad f(x)= \frac{1}{2}\|X^\top~ x-y\|^2,
\end{equation}
where $X=(\bar{x_1}, \bar{x_2},...,\bar{x_n})\in\mathbb{R}^{n\times m}, \bar{x_k}= (x_{k1},x_{k2},...,x_{kn})^\top, ~y=(y_1,y_2,...,y_n)^\top.$\\
 The solution of least square problem \eqref{eq21} is  $x^*= (XX^\top)^{-1}Xy.$ This is hold for\\ $A= XX^\top$ and $b=-Xy$. 

 For quadratic functions, fractional gradient is defined as \begin{align*}
    &~^C_{c} \nabla_{x}^{\alpha} f(x^k) = A x^k + b+ \gamma_{\alpha,\rho} \bar{R}(x^k-c^k),\label{eq-gradient}\\
    ~&\gamma_{\alpha,\rho}= \rho-\frac{1-\alpha}{2-\alpha},~\bar{R}= diag(\sqrt{a_{11}},...,\sqrt{a_{nn}})^\top, \text{and}~ c^k=(0,0,...,0)^\top. \notag
 \end{align*}
To handle some ill-posed matrix and fast convergence, the Tikhonov regularization is employed here. It is also helpful for fast convergence. Equation \eqref{eq21} can be written as Tikhonov regularization problem $(T)$ for any point $\bar{x}$ is 
\begin{equation*}
  (T):\quad  \min_x \|X^\top x-y\|^2 + \gamma \|\bar{R}^\top (x-\bar{x})\|^2,
\end{equation*}
 Here $\bar{R}$ is the Tikhonov matrix, $\bar{x}$ is any point, and $\gamma$ is the Tikhonov parameter for this article. Solution of Tikhonov regularization of problem $(T)$ is as follows
 \begin{equation}\label{eq23}
     x^*_T(\gamma)= \bar{x}+ (XX^\top + \gamma \bar{R}\bar{R}^\top)^{-1} X (y-X^\top \bar{x}),
 \end{equation}
 under mild assumption is matrix $(XX^\top + \gamma \bar{R}\bar{R}^\top)^{-1}$ is invertible.\\
The following theorem establishes convergence results for the Tikhonov regularization applied to a least-squares problem.
\begin{theorem}\label{conv_trik}
Let $f$ be the quadratic function defined in \eqref{eq20}. For $\alpha\in(0,1)$ and\\$\rho\in\mathbb{R}$, define
$
\bar{A}=XX^\top+\gamma_{\alpha,\rho}\bar{R},$
and assume that $\bar{A}$ is positive definite. Any sequence 
$\{x^k\}$ generated by Algorithm~\ref{algorithm 1} for problem $(T)$ converges to the Tikhonov regularization solution \eqref{eq23}.
\end{theorem}

\begin{proof}
Let $c, x\in\mathbb{R}^n$. Let $M = diag\Big([^C_{c_1} D_{x_1}^{\alpha} x_1 \quad \cdots \quad ^C_{c_d} D_{x_d}^{\alpha} x_d ]\Big) \in \mathbb{R}^{n \times n}.$ For $0 < \alpha < 1,~ x^*\in \mathbb{R}^n ~$and $y= X^\top x^*$ it can be checked 
$$
{}^{C}_{c}\nabla^{\alpha}_{x} f_{\alpha,\rho}(x) = M  \Big[X (X^\top x -   {y} )+  \gamma_{\alpha} \,  \bar{R}(x - c) \Big].$$ Similarly, ${}^{C}_{c}\nabla^{\alpha+1}_{x} f_{\alpha,\rho}(x) = M   \Big[ \mathrm{sign}(x - c) \bar{R} \Big]\text{(see Theorem 3.2, \cite{27})}.$\\
Therefore,
\begin{align*}
& {}^{C}_{c}\nabla^{\alpha}_{x} f_{\alpha,\rho}(x) + \rho \, (|x-c|)   {}^{C}_{c}\nabla^{\alpha+1}_{x} f_{\alpha,\rho}(x)=M \Big[ X(X^\top x -   {
y}) + (\rho + \gamma_{\alpha}) \, \bar{R}(x - c) \Big],
\end{align*}
where $\gamma_{\alpha,\rho}= \rho + \gamma_\alpha $ then $d^k$ for $k=0$ is defined as follows

\begin{align}\label{eq24}
&g^k =  X X^\top \Big[ (I+\gamma_{\alpha, \rho}~(XX^\top)^{-1}\bar{R} )x^k -(x^* + \gamma_{\alpha, \rho} ~(XX^\top)^{-1}\bar{R}  ~c)\Big]
\end{align}
 it is obtained from $g^k=0,~x^k =  (I +  \gamma_{\alpha ,\rho}~(XX^\top)^{-1}\bar{R}) ^{-1} 
( x^* + \gamma_{\alpha ,\rho}~c (XX^\top)^{-1}\bar{R} )$.\\
 Observe that,
\begin{align*}
x^k 
&= c + \left(I + \gamma_{\alpha,\beta}(XX^\top)^{-1} \bar{R}\right)^{-1}(x^* - c) \\
&= c + \left(XX^\top + \gamma_{\alpha,\rho} \bar{R}\right)^{-1} XX^\top (x^* - c) \\
&= c + \left(XX^\top + \gamma_{\alpha,\rho} \bar{R}\right)^{-1} X(y - X^\top c) \\
&= x^*_T(\gamma).
\end{align*}
It can be checked from \eqref{eq24},
\begin{align*}
  g^k =  ( X X^\top + \gamma_{\alpha,\rho} \, \bar{R} )(x^k - x^*_T(\gamma)),~~\text {for} ~~ k=0.
\end{align*}
Let $e^k=x^k-x_T^*(\gamma)$ and define the positive definite matrix
$\bar{A}=(XX^\top+\gamma_{\alpha,\rho}\,\bar{R})$. Then $g^k=\bar{A}e^k$.
Moreover,
\begin{align*}
{}^{C}_{c}\nabla_{x}^{\alpha} f(x^{k+1})
&=\bar{A}e^{k+1}\\
&=\bar{A}\big(x^k+\eta_k d^k-x_T^*(\gamma)\big)\\
&=\bar{A}\big(e^k+\eta_k d^k\big).
\end{align*}
By the Wolfe condition,
\begin{align*}
\Big(\bar{A}(e^k+\eta_k d^k)\Big)^\top d^k
\ge c_2\,(\bar{A}e^k)^\top d^k,
\end{align*}
and hence
\begin{align*}
\eta_k \ge \frac{(1-c_2)\big(-\,e^{k\top}\bar{A}d^k\big)}{d^{k\top}\bar{A}d^k}\geq 0.\label{eta_k}
\end{align*}
Remaining is same as Theorem \eqref{convergence thm} that sequence $\{x^k\}$ generated for Problem~$T$ converges to the Tikhonov solution $x_T^*(\gamma)$.

\end{proof}

\section{Numerical example.}\label{sec4}
Numerical examples have been presented to validate the theoretical results and to demonstrate the effectiveness of the proposed methods. The article includes three distinct examples for comparison and validation. The first two are adapted from \cite{27}.

In Example~\ref{example1}, the convergence behaviour of the proposed CFCG method has been compared with Caputo fractional steepest descent (CFSD) method in the context of Tikhonov regularization based on parameters. This example demonstrates the numerical stability and efficiency of the proposed approach. Example~\ref{example2} focuses on function-based learning using neural networks. The performance of both CFSD and CFCG methods has been evaluated on three benchmark functions. Mean of iterations numbers and function evaluation has been reported in a comparative table to highlight the optimization quality. In Example~\ref{example3}, a neural network classifier has been tested on two real-world datasets: Fashion MNIST and KMNIST. This example provides a comprehensive comparison between stochastic sub-gradient descent (SSD) and the proposed CFCG method. Mini-batch suits MNIST and KMIST by delivering low-cost, frequent updates from small random batches, achieving faster time-to-accuracy and lower memory use than full-batch GD. Its inherent stochasticity provides implicit regularization—escaping saddles/sharp minima and favouring flatter minima—thus improving generalization on handwritten-digit classification. As ReLU activation functions, commonly used in classification networks, are non-differentiable at zero, the overall objective becomes non-smooth. Consequently, we work with sub-differential method. Tables, graphs, and visual predictions have been included to support the evaluation. Together, these examples demonstrate the improved convergence, stability, and generalization performance achieved by the CFCG method.

\begin{example}\label{example1}\textbf{Random datasets for Tikhonov parameters:} 

To explore the convergence behaviour of the proposed \text{CFCG} method in comparison with the \text{CFSD} approach for solving Tikhonov regularization problems, we consider quadratic objective functions of the form~\eqref{eq20}.
The data matrix $X \in \mathbb{R}^{m \times n}$ and vector $y \in \mathbb{R}^{m \times 1}$ are generated with entries randomly drawn from the interval $(-1, 1)$.
For this study, we fix the dimensions as $m = n = 100$. From this setup, the matrix $A$ and vector $b$ are defined as $A = X X^\top$ and $b = -A y$, respectively, establishing the optimization landscape.

The optimization is initialized with a vector $x$, where each component is independently sampled from a uniform distribution over $(1, 10)$, ensuring a randomized and diverse starting point.
The vector $c$, representing the lower bound of integration in the Caputo derivative, is set uniformly to one, reflecting the non-local memory characteristic of the fractional formulation.

To improve the conditioning of the least-squares problem, \text{Tikhonov regularization} is applied.
The regularization parameter $\gamma$ is varied within the set ${0.5, 0.75, 1, 2, 3, 4}$.
For each $\gamma$, we track the Euclidean distance to the regularized solution as a function of the iteration count.
As expected from the theoretical convergence analysis, the proposed \text{CFCG} method achieves substantially faster convergence than the \text{CFSD} method.
The numerical results consistently support this observation.

The results presented in the following tables and figures reinforce these findings.
The comparison is based on two key performance metrics: (i) the number of iterations and (ii) the number of function evaluation.
Since the gradient of a quadratic function has the linear form $A x + B$, gradient computations do not contribute to the function-evaluation counter; only evaluations of the objective function and Armijo-Wolfe line-search steps are included.
Following the approach of~\cite{27}, the \text{CFSD} method employs a fixed step size as defined therein. In the steepest descent method, the number of function evaluation exceeds the number of iterations by one, since an additional function value $f(x^0)$ is computed at initialization to start the procedure.
All simulations were implemented in Python using \text{Google Colab}, with a stopping tolerance of $10^{-4}$ and a maximum of 2000 iterations.
The fractional order was fixed at $\alpha = 0.9$.
To ensure fairness, both methods (\text{CFSD} and \text{CFCG}) were executed within the same code cell, eliminating bias due to independently generated random data.

Table~\ref{Table1} presents the results for the $\beta^{FR}$ formulation of the \text{CFCG} method, demonstrating its superior performance for every tested value of $\gamma$ compared with the CFSD method.
The first graph of Figure~\ref{fig1} further illustrates this, showing that CFSD requires substantially more iterations to reach comparable accuracy.

Similarly, the second row of Table~\ref{Table1} and the second graph of Figure~\ref{fig1} report results for the $\beta^{CD}$ variant, again confirming the advantage of the CFCG method. The convergence profile in this figure clearly demonstrates the efficiency of the proposed \text{CFCG} method. The \text{CFSD} method (red curve) exhibits a gradual decline in the objective value, indicating slow progress due to its fixed step size. In contrast, the \text{Proposed CFCG} method (green curve) converges rapidly and smoothly toward the optimum, requiring significantly fewer iterations. This behaviour confirms that incorporating fractional-order memory enhances both the convergence speed and stability of the algorithm.
For the remaining formulations ($\beta^{DY},\beta^{PRP}$ and $\beta^{HS}$), the numerical trends continue to favour the CFCG method, indicating faster convergence and enhanced stability compared with CFSD.
Although figures are provided only for the first two variants, the corresponding results for all four $\beta$ formulations consistently support the superior performance of the proposed fractional-order approach.

\end{example}
\begin{table}[htbp]
\centering
\small
\setlength{\tabcolsep}{3.5pt}
\renewcommand{\arraystretch}{0.9}

\begin{tabular}{|c|c|cc|cc|}
\hline
\multirow{2}{*}{$\beta$ Variant} & \multirow{2}{*}{$\gamma$} 
& \multicolumn{2}{c|}{CFSD} 
& \multicolumn{2}{c|}{CFCG } \\
\cline{3-6}
 &  & Iterations & Func. Eval. & Iterations & Func. Eval. 
 \\
\hline
\multirow{6}{*}{$FR$} 
 & 0.5  & 1427 & 1428 & 40 & 411 \\
 & 0.75 & 1656 & 1657 & 42 & 455 \\
 & 1.0  & 1341 & 1342 & 51 & 548 \\
 & 2.0  & 691  & 692  & 45 & 488 \\
 & 3.0  & 471  & 472  & 35 & 348 \\
 & 4.0  & 377  & 378  & 50 & 528 \\
\hline

\hline
\multirow{6}{*}{$CD$} 
 & 0.5  & 2000 & 2001 & 41 & 434 \\
 & 0.75 & 1572 & 1573 & 93 & 980 \\
 & 1.0  & 1248 & 1249 & 57 & 583 \\
 & 2.0  & 682  & 683  & 39 & 430 \\
 & 3.0  & 478  & 479  & 66 & 708 \\
 & 4.0  & 373  & 374  & 54 & 569 \\
\hline

\hline
\multirow{6}{*}{$DY$} 
 & 0.5  & 1760 & 1761 & 63 & 738 \\
 & 0.75 & 1503 & 1504 & 54 & 661 \\
 & 1.0  & 1251 & 1252 & 52 & 628 \\
 & 2.0  & 647  & 648  & 46 & 519 \\
 & 3.0  & 486  & 487  & 82 & 1146 \\
 & 4.0  & 362  & 363  & 44 & 588 \\
\hline

\hline
\multirow{6}{*}{$PRP$} 
 & 0.5  & 1845 & 1846 & 51 & 629 \\
 & 0.75 & 1740 & 1741 & 56 & 623 \\
 & 1.0  & 1347 & 1348 & 79 & 886 \\
 & 2.0  & 731  & 732  & 91 & 1232 \\
 & 3.0  & 484  & 485  & 76 & 1010 \\
 & 4.0  & 386  & 387  & 72 & 917 \\
\hline

\hline
\multirow{6}{*}{$HS$} 
 & 0.5  & 1980 & 1981 & 28 & 313 \\
 & 0.75 & 1507 & 1508 & 23 & 242 \\
 & 1.0  & 1106 & 1107 & 23 & 242 \\
 & 2.0  & 615  & 616  & 41 & 457 \\
 & 3.0  & 444  & 445  & 37 & 391 \\
 & 4.0  & 340  & 341  & 37 & 291 \\
\hline

\hline
\end{tabular}
\caption{Comparison between CFSD and CFCG.}
\label{Table1}
\end{table}
\begin{figure}[H]
  \centering
  \begin{subfigure}[b]{0.47\textwidth}
   \includegraphics[width=\linewidth,height=2.40cm]{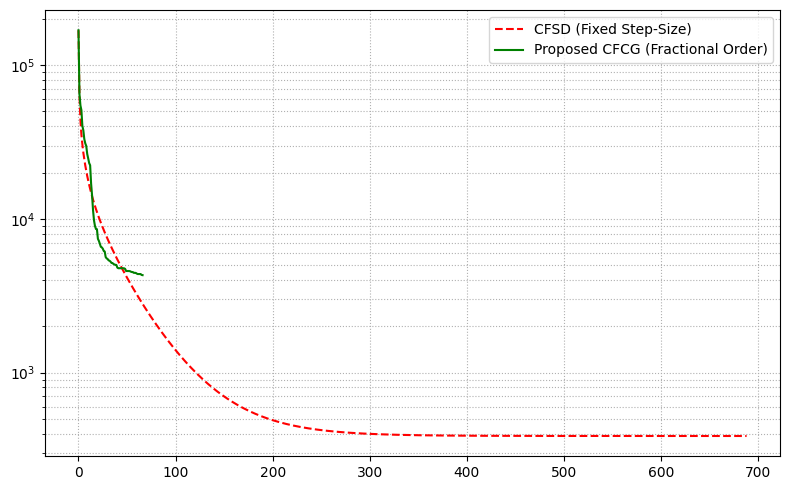}
  \label{fig11}
  \end{subfigure}
  \hfill
    \begin{subfigure}[b]{0.47\textwidth}
    \includegraphics[width=\linewidth,height=2.40cm]{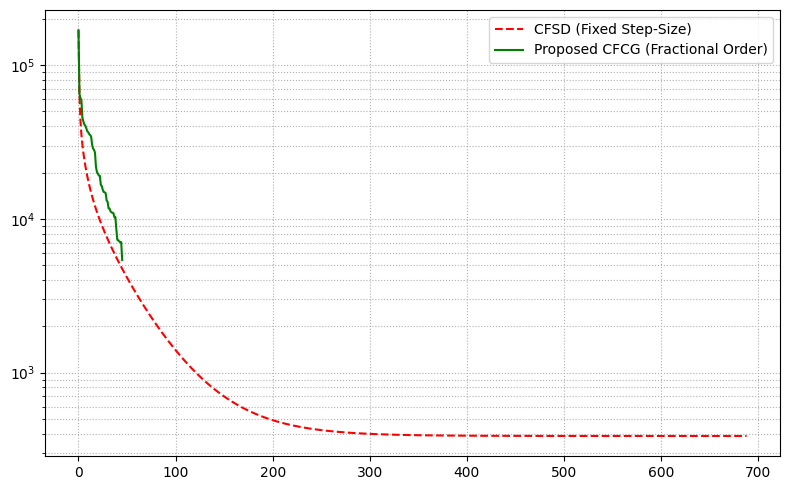}
    \label{fig12}
  \end{subfigure}
    \caption{Convergences result of CFSD and CFCG methods for $\beta^{FR}$ and $\beta^{CD}$.}
  \label{fig1}

\end{figure}
In conclusion, the numerical example proposed by CFCG based on five types of $\beta$ performs better than CFSD in terms of function evaluation and iteration number. Table \ref{Table1} and Figure \ref{fig1} are present here to support.
\begin{example}\label{example2}\textbf{Neural networks:}

In this second numerical experiment, we compare the performance of CFSD and CFCG for training a neural network to approximate three benchmark functions. These functions are defined as
\begin{itemize}
\item[1.]$h_1(z) = \sin(5\pi z)$
\item[2.]$h_2(z) = \sin(2\pi z)\,\exp(-z^2)$
\item[3.]$h_3(z) = \mathbb{1}_{\{z>0\}}(z) + 0.2 \sin(2\pi z)$
\end{itemize}
where the indicator function \( \mathbb{1}_{\{z > 0\}}(z) \) is given by
\[
\mathbb{1}_{\{z > 0\}}(z) = 
\begin{cases} 
1, & \text{if } z > 0, \\
0, & \text{otherwise}.
\end{cases}
\]
The goal of this experiment is to assess the convergence speed, in terms of iteration counts and function evaluation of the two optimization methods when training a single-layer feedforward neural network with 60 hidden units and a \(\tanh\) activation function to approximate these three target functions.

To carry out this comparison, we employ CFCG and CFSD as the optimization procedures for updating the network parameters. CFCG incorporates fractional gradients with memory, allowing for a more flexible update rule that adapts to the loss function's curvature over time, while CFSD uses the Caputo fractional steepest descent approach but with fractional gradients and memory to improve convergence. The CFSD method is implemented with adaptive learning rate selection, where the learning rate is chosen from the grid \(\{0.01, 0.005, 0.001, 0.0005, 0.0001, 0.00005, 0.00001\}\) and Armijo-Wolfe backtracking is employed to ensure stable updates for the CFCG method. For the experiment, we trained the neural network on each of the three target functions. Each function was trained over five trials, each with 100 training points uniformly sampled from the interval \([-1,1]\). The training process was limited to a maximum of 2000 iterations, or until the loss converged to a tolerance of \(10^{-4}\). The fractional order \(\alpha\) was varied across the values \(0.5\), \(0.6\), \(0.7\), \(0.8\), and \(0.9\), and the performance of each optimization method was measured in terms of the number of iterations and function evaluation required to achieve convergence, which is presented in following table.

\begin{table}[htbp]
\centering
\small
\setlength{\tabcolsep}{4pt}
\renewcommand{\arraystretch}{1.15}

\begin{tabular}{c l ccccc c}
\toprule
\multirow{2}{*}{$\alpha$} & \multirow{2}{*}{Metric} &
\multicolumn{5}{c}{CFCG ($\beta$-type)} & CFSD \\
\cmidrule(lr){3-7}
 & & FR & CD & DY & PRP & HS &  \\
\midrule
\multirow{2}{*}{0.5} & Iter. & 4.733 & 3.200 & 6.467 & 6.467 & 6.467 & 33.867 \\
                    & Func.Eva.    & 8398.40 & 7240.00 & 5405.86 & 6069.533 & 5249.00 & 6581.80 \\
\midrule
\multirow{2}{*}{0.6} & Iter. & 5.600 & 3.667 & 6.067 & 6.067 & 6.067 & 39.267 \\
                    & Func.Eva.    & 8060.53 & 8000.20 & 5055.93 & 7481.33  & 6093.66 & 7602.40 \\
\midrule
\multirow{2}{*}{0.7} & Iter. & 4.533 & 3.667 & 4.333 & 4.333 & 4.333 & 41.800 \\
                    &Func.Eva.     & 5128.33 & 6926.26 & 6057.46 & 6697.00  & 5719.60 & 8081.20 \\
\midrule
\multirow{2}{*}{0.8} & Iter. & 3.800 & 2.533 & 5.467 & 5.467 & 5.467 & 46.867 \\
                    &Func.Eva.    & 6781.46 & 5659.26 & 5224.86 & 7119.33  & 5659.26 & 9038 \\
\midrule
\multirow{2}{*}{0.9} & Iter. & 5.067 & 3.667 & 6.400 & 6.400 & 6.400 & 31.800 \\
                    & Func.Eva.     & 8338.06 & 7288.26 & 5852.33 & 6612 & 6286.73 & 6191 \\
\bottomrule
\end{tabular}

\caption{Mean iterations and mean function evaluation for different values of $\alpha$ for CFCG (with various $\beta$-types) and for CFSD.}
\label{tab:mean-iters-alpha-beta}
\end{table}

\begin{figure}[ht]
    \centering
    \begin{subfigure}[b]{0.45\textwidth}
   \includegraphics[width=\linewidth,height=2.30cm]{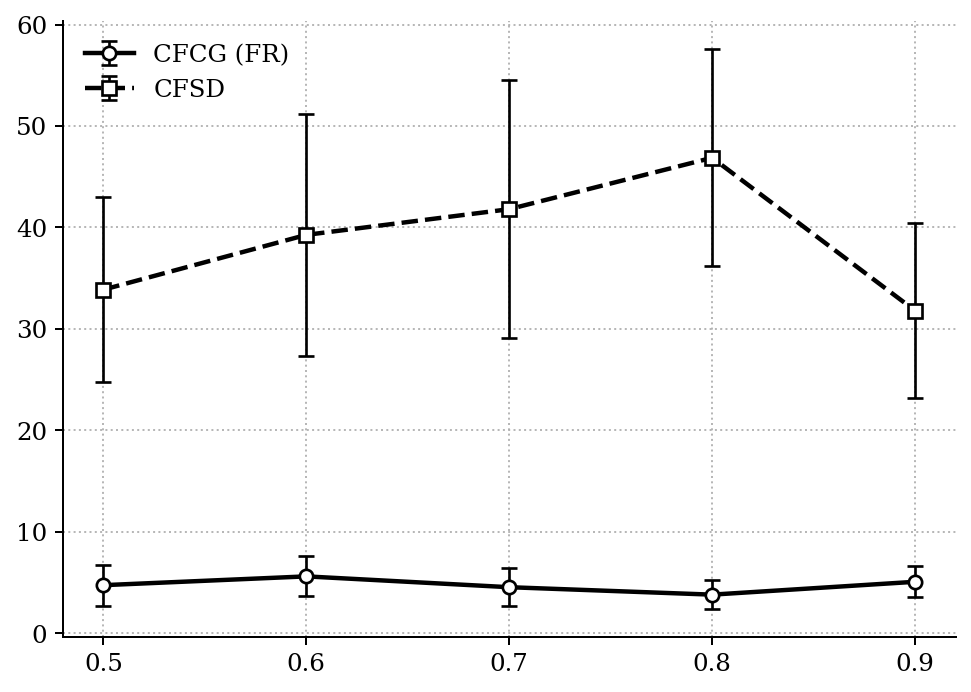}
        \end{subfigure}
    \hfill
    \begin{subfigure}[b]{0.45\textwidth}
   \includegraphics[width=\linewidth,height=2.30cm]{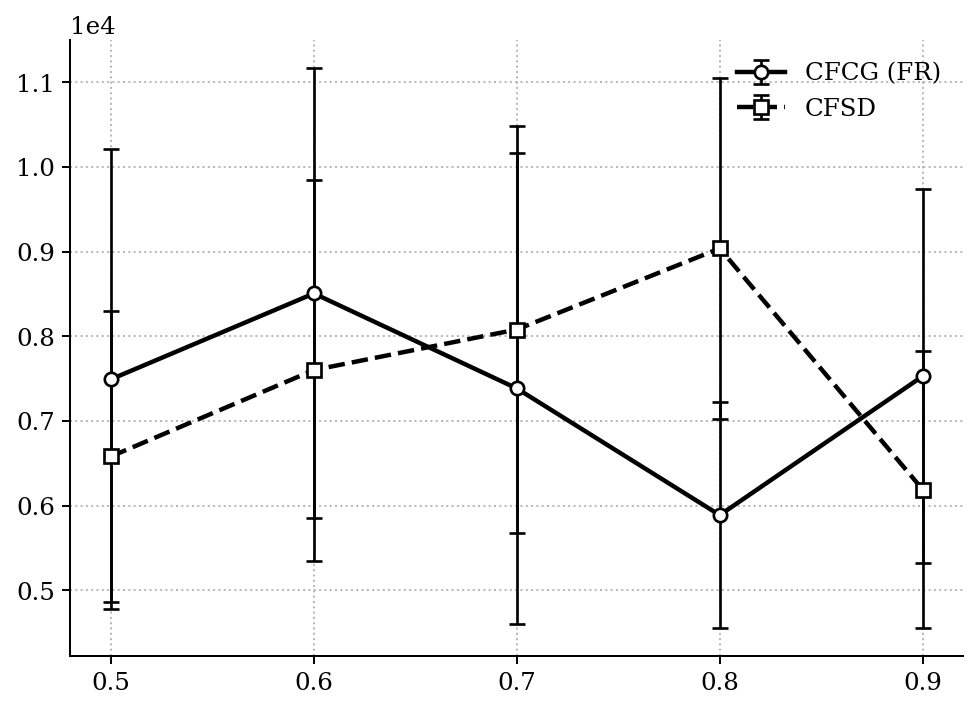}    
         \end{subfigure}
    \hfill
    \caption{Iterations and function evaluation for CFSD and CFCG of $\beta^{FR}_k$ .}
    \label{fig6}
\end{figure}
\begin{figure}[ht]
    \centering
    \begin{subfigure}[b]{0.45\textwidth}
   \includegraphics[width=\linewidth,height=2.30cm]{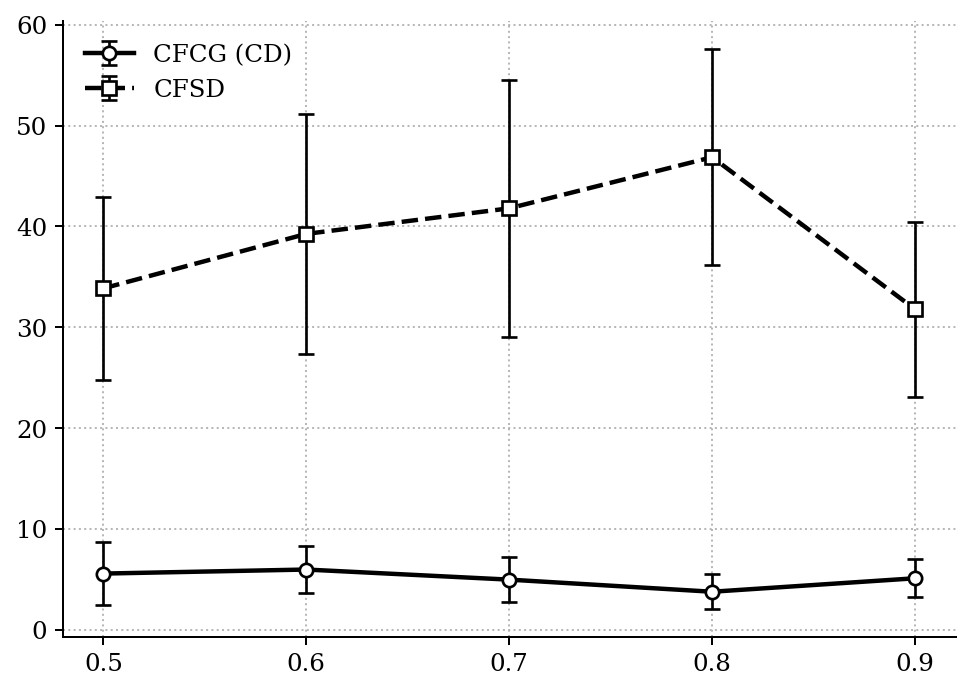}
        \end{subfigure}
    \hfill
    \begin{subfigure}[b]{0.45\textwidth}
   \includegraphics[width=\linewidth,height=2.30cm]{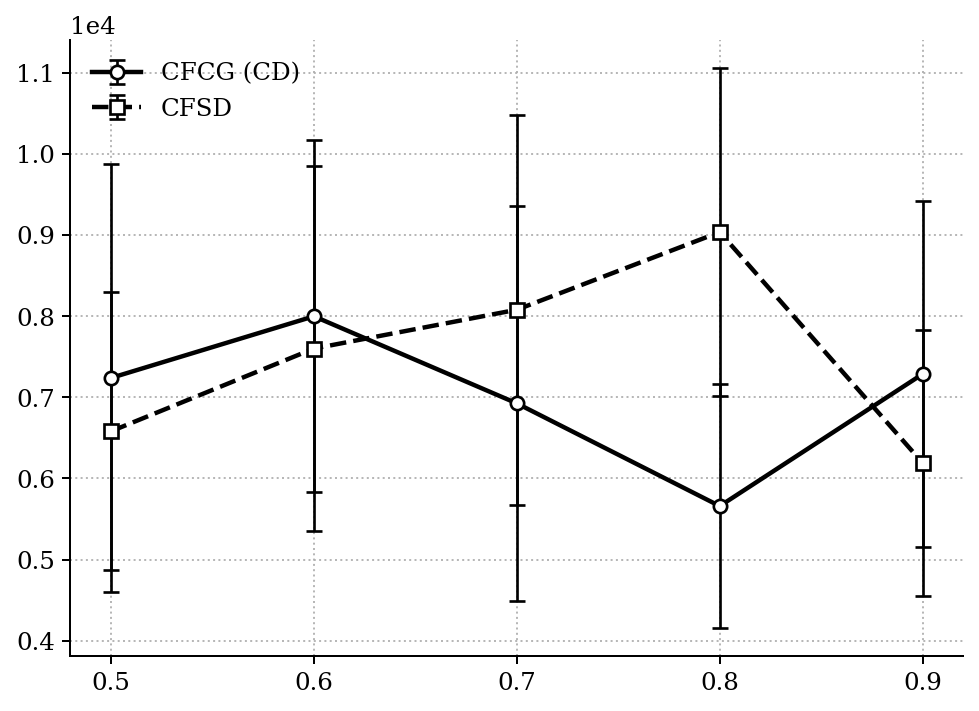}    
         \end{subfigure}
    \hfill
    \caption{Iterations and function evaluation for CFSD and CFCG of $\beta^{CD}_k$ .}
    \label{fig7}
\end{figure}

 Table~\ref{tab:mean-iters-alpha-beta} reports the mean number of iterations and function evaluations required by CFCG and CFSD to reach the stopping criterion. The results are shown for different fractional orders $\alpha$ and for the five classical $\beta$-update formulas $(FR, CD, DY, PRP, HS)$. For all choices of $\alpha$ and $\beta$, the proposed CFCG method converges in very few iterations, typically between about $2.5$ and $6.5$. In contrast, CFSD needs about $32$ to $47$ iterations. Depending on $\alpha$ and the $\beta$-type, this gives roughly a $5$--$18$ times reduction in iteration count compared with CFSD. Among the CFCG variants, the $CD$-type $\beta$ update is consistently the most efficient. It gives the smallest mean iteration numbers for every $\alpha$, with the best performance at $\alpha=0.8$. The $FR$-type $\beta$ update is slightly less aggressive. The $DY$, $PRP$, and $HS$ variants give similar and mildly larger iteration counts, but they are still far below those of CFSD. The dependence on $\alpha$ is moderate. The advantage of CFCG over CFSD is robust for $0.5\le \alpha\le 0.9$. This indicates that fractional-order conjugate-gradient directions can significantly accelerate convergence compared with the fractional steepest-descent baseline. For function evaluations, CFSD uses an adaptive learning rate, while CFCG uses an Armijo--Wolfe line search. As a result, in many cases CFCG takes more function evaluations than CFSD. Figures~\ref{fig6} and \ref{fig7} are included to support these results.

In conclusion, this study compared CFCG and CFSD for training a neural network to approximate three benchmark functions. It shows that CFCG consistently outperformed CFSD in convergence speed, requiring fewer iterations to reach comparable or better test errors. The iteration count is small, but each iteration can require many function evaluations. This is because we use finite-difference gradients and a line search, which call the function many times. The results suggest that the Caputo Fractional Conjugate Gradient method is a promising optimization technique for faster convergence, particularly with fractional gradient updates. Therefore, CFCG can be a useful candidate for improving the efficiency of neural network training, especially when the loss function is complex and contains non-smooth regions.

\end{example}
\begin{example}\label{example3}\textbf{Neural networks classifier:}
To demonstrate the performance of the proposed method, experiments were conducted on two benchmark datasets: Fashion MNIST \cite{f-MNIST} and KMNIST \cite{KMNIST}. Both datasets consist of grayscale images of size $28 \times 28$ pixels with $10$ class labels. All pixel values were normalized to the range $[-1, 1]$ using a mean and standard deviation of $0.5$. The datasets were processed as follows.

\begin{itemize}
    \item \text{Fashion MNIST:} Contains $60{,}000$ training and $10{,}000$ test images of fashion items (e.g., T-shirt, sneaker, bag).
    \item \text{KMNIST:} Comprises $60{,}000$ training and $10{,}000$ test images of handwritten Hiragana characters. For a subset of experiments, the datasets were restricted to the first five classes (labels $0$--$4$), resulting in a simplified $5$-class classification task.
\end{itemize}

The classification model used in all experiments is a single-hidden-layer fully connected neural network with ReLU activation. The structure is specified as follows.

\begin{itemize}
    \item Input: Flattened $28 \times 28$ grayscale image ($784$ features).
    \item Hidden layer: $128$ units, followed by ReLU activation.
    \item Output layer: $10$ units (corresponding to the number of classes).
\end{itemize}

In this numerical experiment, a comparison was made between SSD and the proposed CFCG algorithm. All models were trained for $5$ epochs, which means that each training sample was processed five times. A batch size of $64$ was used, and the standard cross-entropy loss function was minimized for multiclass classification. The experiments were implemented in Python using PyTorch and executed in the Google Colab environment. Core PyTorch modules such as \texttt{torch.nn}, \texttt{torch.utils.data}, and \texttt{torchvision} were utilized for model construction and data handling. The SSD method works stochastically, whereas CFCG works for all datasets differently in each epoch.

During training, the loss and test accuracy were recorded in each epoch for both optimizers. These metrics were visualized using Matplotlib. Two subplots display the evolution of training loss and test accuracy versus epoch for both methods. In addition, a prediction visualization was included: ten randomly selected test images were shown along with their true class and the predictions made by both SSD and CFCG. To further illustrate performance differences, test images were identified in which SSD produced incorrect predictions while CFCG produced correct ones. These cases were also visualized with both predictions and the true label, enabling qualitative analysis of the improvements provided by the proposed method.

This section presents the performance comparison between SSD and CFCG across five $\beta$ strategies. Tables~\ref{Table:3} report training loss and test accuracy over five epochs on Fashion MNIST and KMNIST datasets. For the $FR$ configuration, CFCG achieved $86.20\%$ accuracy on Fashion MNIST and $90.12\%$ on KMNIST by the 5\textsuperscript{th} epoch, while SSD reached only $82.10\%$ and $80.54\%$, respectively. The $CD$ variant further amplified this gain, with CFCG attaining $86.85\%$ accuracy on Fashion MNIST and $89.98\%$ on KMNIST, compared to SSD's $76.11\%$ and $74.62\%$. In the $DY$ setting, CFCG reached $82.10\%$ on Fashion MNIST and $82.18\%$ on KMNIST, whereas SSD lagged at $76.11\%$ and $74.62\%$. For PRP, CFCG recorded $81.82\%$ (Fashion MNIST) and $82.54\%$ (KMNIST), in contrast to SSD’s $76.11\%$ and $74.62\%$. The $HS$ strategy also favoured CFCG, yielding $82.38\%$ on Fashion MNIST and $81.94\%$ on KMNIST, while SSD remained at $76.11\%$ and $74.62\%$.

Figures~\ref{KMNIST_1_FR_} and \ref{KMNIST_1_CD_} present the convergence plots for Fashion MNIST. In each case, the loss curve for CFCG decreased more sharply, and the accuracy curve consistently surpassed that of SSD. Similar trends were observed on KMNIST in Figures~\ref{KMNIST_1_FR_}, \ref{KMNIST_1_CD_}, where CFCG stabilized more quickly and outperformed SSD in all epochs. Visual predictions further highlight these differences. Figures~\ref{KMNIST_2_FR_}, \ref{KMNIST_2_CD_}, \ref{KMNIST_2_DY_}, \ref{KMNIST_2_PPR_}, and \ref{KMNIST_2_HS_} display qualitative output comparisons for Fashion MNIST and KMNIST, where CFCG provided more accurate and visually consistent class predictions even more pronounced improvements, especially in challenging handwritten character recognition.

The KMNIST data set is considered more difficult due to inherent ambiguities in handwritten character shapes. Nevertheless, robust classification performance and generalization were consistently demonstrated by CFCG. In conclusion, the effectiveness of the CFCG method was validated through the results shown in Table~\ref{Table:3} and Figures~\ref{KMNIST_1_FR_}--\ref{KMNIST_2_HS_}. Superior convergence and accuracy across multiple update strategies and datasets were achieved consistently. All corresponding tables, convergence plots, and prediction visualizations have been provided below to support this comparative evaluation.
   \begin{table}[htbp]
\centering

\begin{tabular}{|c|c|c|c|}
\hline
\multirow{2}{*}{$\beta$} & \multirow{2}{*}{Epoch} & \multicolumn{1}{c|}{MNIST (SSD / CFCG)} & \multicolumn{1}{c|}{KMNIST (SSD / CFCG)} \\
\cline{3-4}
 & & Loss / Acc. & Loss / Acc. \\
\hline
\multirow{5}{*}{$FR$}
 & 1 & 0.9878 / 76.21 | 0.5716 / 82.95 & 0.8301 / 74.32 | 0.4114 / 83.30 \\
 & 2 & 0.6051 / 79.22 | 0.4195 / 84.70 & 0.4215 / 77.28 | 0.2281 / 85.74 \\
 & 3 & 0.5385 / 80.64 | 0.3816 / 85.23 & 0.3650 / 79.70 | 0.1723 / 88.18 \\
 & 4 & 0.5026 / 81.62 | 0.3553 / 85.58 & 0.3378 / 80.14 | 0.1403 / 89.24 \\
 & 5 & 0.4797 / 82.10 | 0.3375 / 86.20 & 0.3196 / 80.54 | 0.1156 / 90.12 \\
\hline
\multirow{5}{*}{$CD$}
 & 1 & 1.6704 / 68.16 | 0.5891 / 82.63 & 1.4702 / 67.76 | 0.4267 / 84.02 \\
 & 2 & 1.0259 / 72.37 | 0.4191 / 84.19 & 0.8609 / 70.88 | 0.2269 / 86.74 \\
 & 3 & 0.8261 / 74.05 | 0.3810 / 85.25 & 0.6569 / 72.40 | 0.1735 / 88.22 \\
 & 4 & 0.7376 / 75.14 | 0.3547 / 85.19 & 0.5579 / 73.72 | 0.1407 / 89.00 \\
 & 5 & 0.6863 / 76.11 | 0.3380 / 86.85 & 0.5008 / 74.62 | 0.1205 / 89.98 \\
\hline
\multirow{5}{*}{$DY$}
 & 1 & 1.6704 / 68.16 | 0.9604 / 75.50 & 1.4702 / 67.76 | 0.7721 / 75.36 \\
 & 2 & 1.0259 / 72.37 | 0.5925 / 79.86 & 0.8609 / 70.88 | 0.4044 / 78.42 \\
 & 3 & 0.8261 / 74.05 | 0.5267 / 81.08 & 0.6569 / 72.40 | 0.3505 / 80.62 \\
 & 4 & 0.7376 / 75.14 | 0.4923 / 81.84 & 0.5579 / 73.72 | 0.3215 / 81.16 \\
 & 5 & 0.6863 / 76.11 | 0.4707 / 82.10 & 0.5008 / 74.62 | 0.3005 / 82.18 \\
\hline
\multirow{5}{*}{PRP}
 & 1 & 1.6704 / 68.16 | 0.9501 / 74.67 & 1.4702 / 67.76 | 0.6421 / 75.20 \\
 & 2 & 1.0259 / 72.37 | 0.6173 / 78.92 & 0.8609 / 70.88 | 0.4058 / 76.90 \\
 & 3 & 0.8261 / 74.05 | 0.5484 / 79.61 & 0.6569 / 72.40 | 0.3584 / 79.34 \\
 & 4 & 0.7376 / 75.14 | 0.5136 / 81.44 & 0.5579 / 73.72 | 0.3296 / 81.58 \\
 & 5 & 0.6863 / 76.11 | 0.4934 / 81.82 & 0.5008 / 74.62 | 0.3062 / 82.54 \\
\hline
\multirow{5}{*}{$HS$}
 & 1 & 1.6704 / 68.16 | 0.9574 / 75.69 & 1.4702 / 67.76 | 0.7669 / 75.14 \\
 & 2 & 1.0259 / 72.37 | 0.5942 / 79.96 & 0.8609 / 70.88 | 0.4061 / 78.48 \\
 & 3 & 0.8261 / 74.05 | 0.5279 / 80.99 & 0.6569 / 72.40 | 0.3514 / 80.52 \\
 & 4 & 0.7376 / 75.14 | 0.4940 / 81.65 & 0.5579 / 73.72 | 0.3228 / 80.84 \\
 & 5 & 0.6863 / 76.11 | 0.4727 / 82.38 & 0.5008 / 74.62 | 0.3015 / 81.94 \\
\hline
\end{tabular}
\caption{ Comparison of SSD and CFCG methods across different $\beta$ of Loss / Accuracy (\%) per epoch for MNIST and KMNIST datasets.}
\label{Table:3}
\end{table}

\newpage
\begin{figure}[htbp]
  \centering
  \begin{subfigure}[b]{0.45\textwidth}
   \includegraphics[width=\linewidth,height=2.20cm]{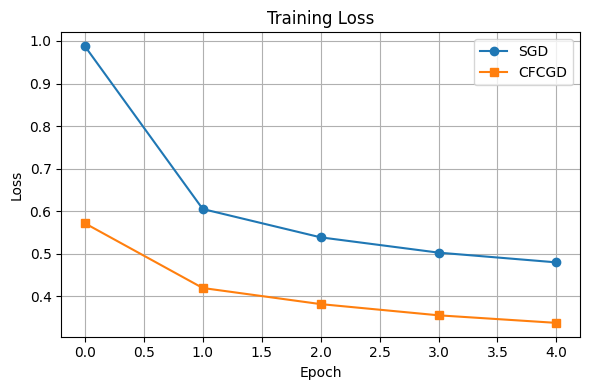}

    \label{MNIST_1}
  \end{subfigure}
  \hfill
    \begin{subfigure}[b]{0.45\textwidth}
    \includegraphics[width=\linewidth,height=2.20cm]{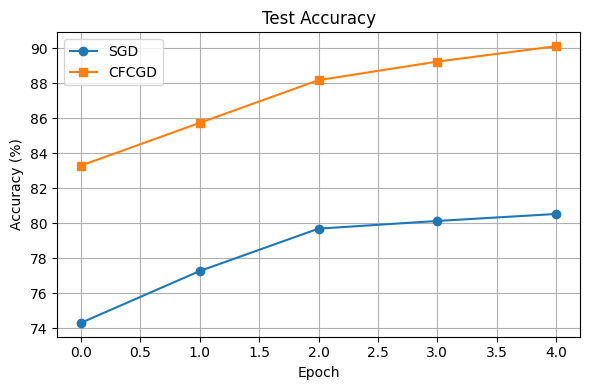}

  \end{subfigure}
  \hfill

  \begin{subfigure}[b]{0.45\textwidth}
   \includegraphics[width=\linewidth,height=2.20cm]{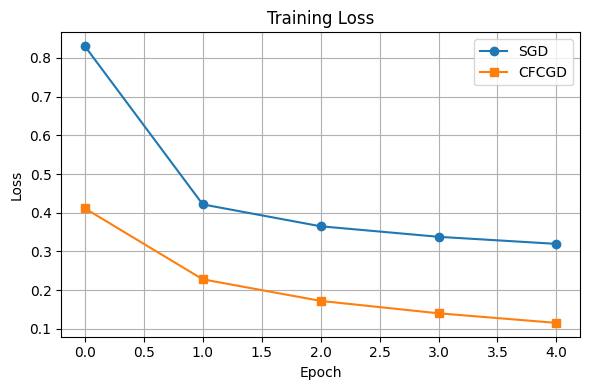}
   
  \end{subfigure}
  \hfill
    \begin{subfigure}[b]{0.45\textwidth}
    \includegraphics[width=\linewidth,height=2.20cm]{F-R_accuracy_KMNIST.png}
 
  \end{subfigure}
    \caption{A convergence behaviour between CFCG and SSD methods for $\beta^{FR}$ of fasion MNIST and KMNIST datasets.}
  \label{KMNIST_1_FR_}

\end{figure}
\hfill
\begin{figure}[!htbp]
  \centering

  \begin{subfigure}[b]{0.24\textwidth}
    \includegraphics[width=\linewidth,height=3cm]{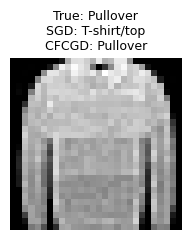}

  \end{subfigure}
  \hfill
  \begin{subfigure}[b]{0.24\textwidth}
    \includegraphics[width=\linewidth,height=3cm]{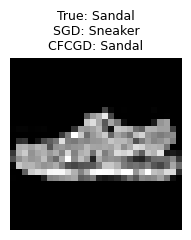}
   
  \end{subfigure}
  \hfill
  \begin{subfigure}[b]{0.24\textwidth}
    \includegraphics[width=\linewidth,height=3cm]{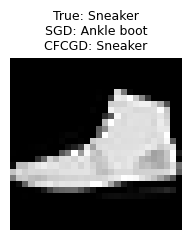}
   
  \end{subfigure}
  \hfill
  \begin{subfigure}[b]{0.24\textwidth}
    \includegraphics[width=\linewidth,height=3cm]{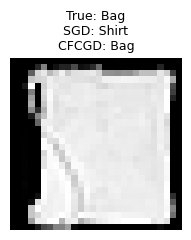}
   
  \end{subfigure}


  \begin{subfigure}[b]{0.24\textwidth}
    \includegraphics[width=\linewidth,height=3cm]{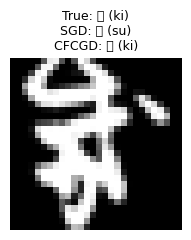}
    
  \end{subfigure}
  \hfill
  \begin{subfigure}[b]{0.24\textwidth}
    \includegraphics[width=\linewidth,height=3cm]{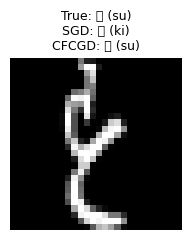}
  
  \end{subfigure}
  \hfill
  \begin{subfigure}[b]{0.24\textwidth}
    \includegraphics[width=\linewidth,height=3cm]{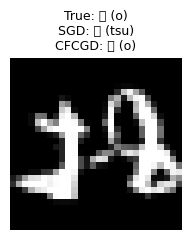}
 
  \end{subfigure}
  \hfill
  \begin{subfigure}[b]{0.24\textwidth}
    \includegraphics[width=\linewidth,height=3cm]{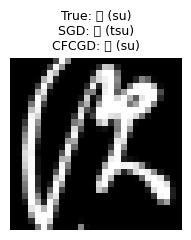}

  \end{subfigure}

  \caption{Visualization of fashion MNIST and KMNIST of $\beta^{FR}$.}
  \label{KMNIST_2_FR_}
\end{figure}



\begin{figure}[htbp]
  \centering
  \begin{subfigure}[b]{0.45\textwidth}
   \includegraphics[width=\linewidth,height=2.20cm]{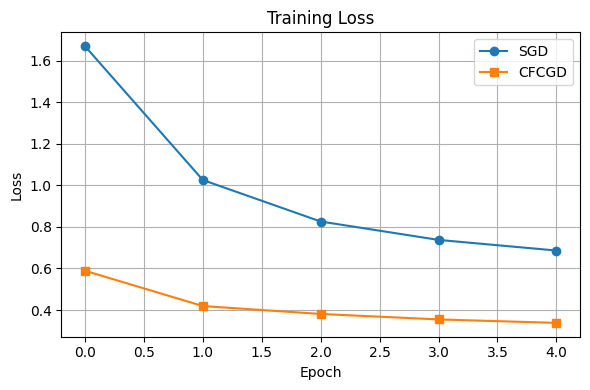}

  \end{subfigure}
  \hfill
    \begin{subfigure}[b]{0.45\textwidth}
    \includegraphics[width=\linewidth,height=2.20cm]{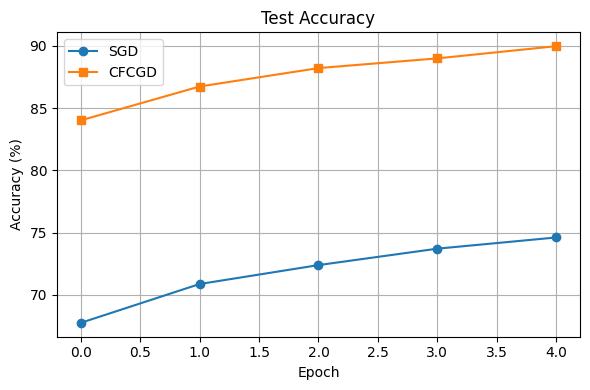}

   \end{subfigure}
   \begin{subfigure}[b]{0.45\textwidth}
   \includegraphics[width=\linewidth,height=2.20cm]{C-D_loss_MNIST.png}
   
  \end{subfigure}
  \hfill
    \begin{subfigure}[b]{0.45\textwidth}
    \includegraphics[width=\linewidth,height=2.20cm]{C-D_accuracy_KMNIST.png}
 
  \end{subfigure}
    \caption{A convergence behaviour between CFCG and SSD methods for $\beta^{CD}$ of fashion MNIST and KMNIST datasets.}
  \label{KMNIST_1_CD_}

\end{figure}
\hfill
\begin{figure}[H]
  \centering

  \begin{subfigure}[b]{0.24\textwidth}
    \includegraphics[width=\linewidth,height=3cm]{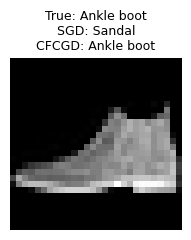}

  \end{subfigure}
  \hfill
  \begin{subfigure}[b]{0.24\textwidth}
    \includegraphics[width=\linewidth,height=3cm]{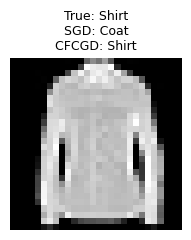}
   
  \end{subfigure}
  \hfill
  \begin{subfigure}[b]{0.24\textwidth}
    \includegraphics[width=\linewidth,height=3cm]{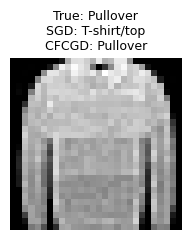}
   
  \end{subfigure}
  \hfill
  \begin{subfigure}[b]{0.24\textwidth}
    \includegraphics[width=\linewidth,height=3cm]{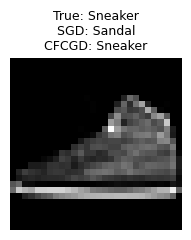}
   
  \end{subfigure}



  \begin{subfigure}[b]{0.24\textwidth}
    \includegraphics[width=\linewidth,height=3cm]{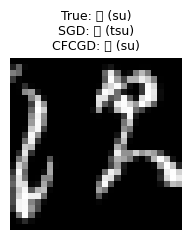}
    
  \end{subfigure}
  \hfill
  \begin{subfigure}[b]{0.24\textwidth}
    \includegraphics[width=\linewidth,height=3cm]{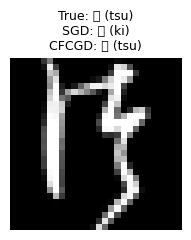}
  
  \end{subfigure}
  \hfill
  \begin{subfigure}[b]{0.24\textwidth}
    \includegraphics[width=\linewidth,height=3cm]{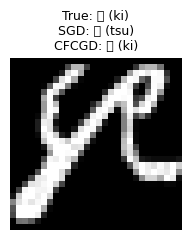}
 
  \end{subfigure}
  \hfill
  \begin{subfigure}[b]{0.24\textwidth}
    \includegraphics[width=\linewidth,height=3cm]{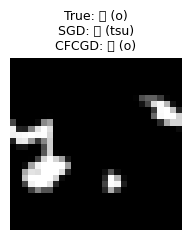}

  \end{subfigure}

  \caption{Visualization of fashion MNIST and KMNIST of $\beta^{CD}$.}
  \label{KMNIST_2_CD_}
\end{figure}

\begin{figure}[!htbp]
  \centering

  \begin{subfigure}[b]{0.24\textwidth}
    \includegraphics[width=\linewidth,height=3cm]{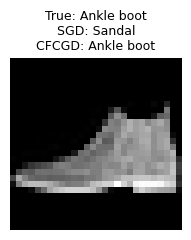}

  \end{subfigure}
  \hfill
  \begin{subfigure}[b]{0.24\textwidth}
    \includegraphics[width=\linewidth,height=3cm]{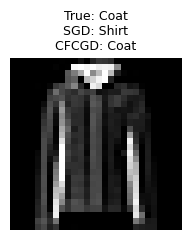}
   
  \end{subfigure}
  \hfill
  \begin{subfigure}[b]{0.24\textwidth}
    \includegraphics[width=\linewidth,height=3cm]{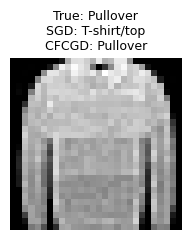}
   
  \end{subfigure}
  \hfill
  \begin{subfigure}[b]{0.24\textwidth}
    \includegraphics[width=\linewidth,height=3cm]{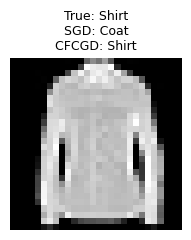}
   
  \end{subfigure}



  \begin{subfigure}[b]{0.24\textwidth}
    \includegraphics[width=\linewidth,height=3cm]{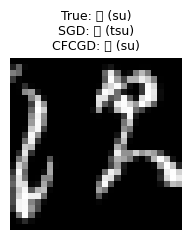}

  \end{subfigure}
  \hfill
  \begin{subfigure}[b]{0.24\textwidth}
    \includegraphics[width=\linewidth,height=3cm]{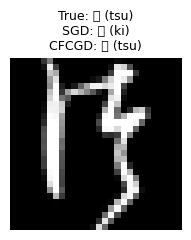}
   
  \end{subfigure}
  \hfill
  \begin{subfigure}[b]{0.24\textwidth}
    \includegraphics[width=\linewidth,height=3cm]{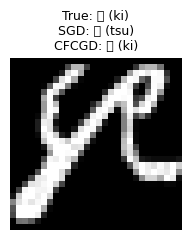}
   
  \end{subfigure}
  \hfill
  \begin{subfigure}[b]{0.24\textwidth}
    \includegraphics[width=\linewidth,height=3cm]{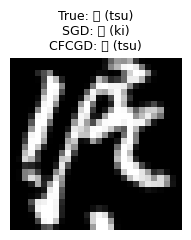}
   
  \end{subfigure}

  \caption{Visualization of fashion MNIST and KMNIST of $\beta^{DY}$.}
  \label{KMNIST_2_DY_}
\end{figure}
\begin{figure}[!htbp]
  \centering

  \begin{subfigure}[b]{0.24\textwidth}
    \includegraphics[width=\linewidth,height=3cm]{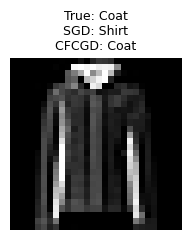}

  \end{subfigure}
  \hfill
  \begin{subfigure}[b]{0.24\textwidth}
    \includegraphics[width=\linewidth,height=3cm]{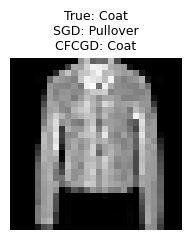}
   
  \end{subfigure}
  \hfill
  \begin{subfigure}[b]{0.24\textwidth}
    \includegraphics[width=\linewidth,height=3cm]{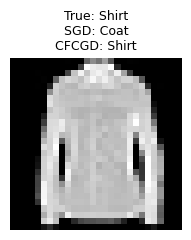}
   
  \end{subfigure}
  \hfill
  \begin{subfigure}[b]{0.24\textwidth}
    \includegraphics[width=\linewidth,height=3cm]{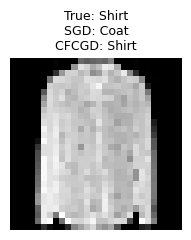}
   
  \end{subfigure}


  \begin{subfigure}[b]{0.24\textwidth}
    \includegraphics[width=\linewidth,height=3cm]{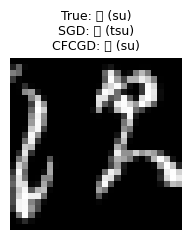}

  \end{subfigure}
  \hfill
  \begin{subfigure}[b]{0.24\textwidth}
    \includegraphics[width=\linewidth,height=3cm]{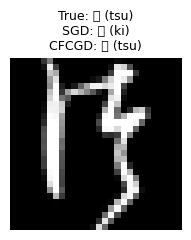}
   
  \end{subfigure}
  \hfill
  \begin{subfigure}[b]{0.24\textwidth}
    \includegraphics[width=\linewidth,height=3cm]{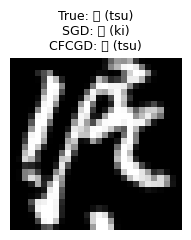}
   
  \end{subfigure}
  \hfill
  \begin{subfigure}[b]{0.24\textwidth}
    \includegraphics[width=\linewidth,height=3cm]{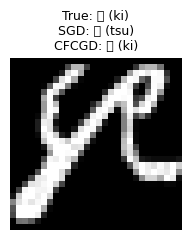}
   
  \end{subfigure}

  \caption{Visualization of fashion MNIST and KMNIST of $\beta^{PRP}$.}
  \label{KMNIST_2_PPR_}
\end{figure}
\begin{figure}[H]
  \centering

  \begin{subfigure}[b]{0.24\textwidth}
    \includegraphics[width=\linewidth,height=3cm]{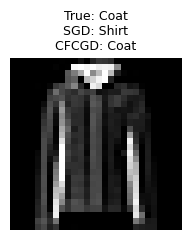}

  \end{subfigure}
  \hfill
  \begin{subfigure}[b]{0.24\textwidth}
    \includegraphics[width=\linewidth,height=3cm]{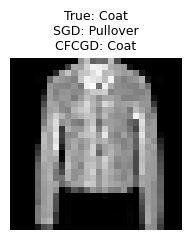}
   
  \end{subfigure}
  \hfill
  \begin{subfigure}[b]{0.24\textwidth}
    \includegraphics[width=\linewidth,height=3cm]{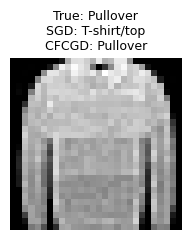}
   
  \end{subfigure}
  \hfill
  \begin{subfigure}[b]{0.24\textwidth}
    \includegraphics[width=\linewidth,height=3cm]{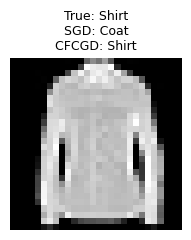}
   
  \end{subfigure}



  \begin{subfigure}[b]{0.24\textwidth}
    \includegraphics[width=\linewidth,height=3cm]{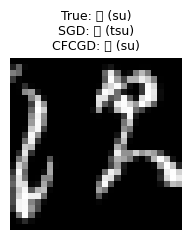}

  \end{subfigure}
  \hfill
  \begin{subfigure}[b]{0.24\textwidth}
    \includegraphics[width=\linewidth,height=3cm]{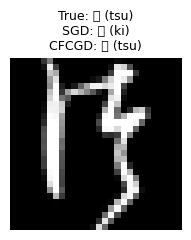}
   
  \end{subfigure}
  \hfill
  \begin{subfigure}[b]{0.24\textwidth}
    \includegraphics[width=\linewidth,height=3cm]{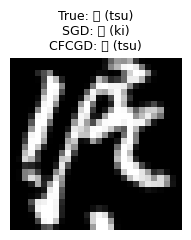}
   
  \end{subfigure}
  \hfill
  \begin{subfigure}[b]{0.24\textwidth}
    \includegraphics[width=\linewidth,height=3cm]{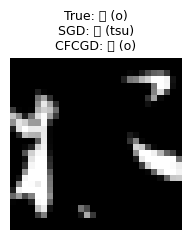}
   
  \end{subfigure}

  \caption{Visualization of fashion MNIST and KMNIST of $\beta^{HS}$.}
  \label{KMNIST_2_HS_}
\end{figure}
\end{example}
\begin{remark}
Numerous choices for \(\beta\) have been proposed; here, five are considered. Performance is application dependent, and no single \(\beta\) can be expected to be uniformly superior across problems, even when descent is ensured.
\end{remark}

\section{Conclusion}\label{sec5}
We have proposed a CFCG method for solving unconstrained single-objective optimization problems. The method is applicable to both smooth and non-smooth functions and is capable of incorporating classical derivative information. Five different forms of the CG parameter \( \beta \) are explored to enhance convergence. The descent property of each direction is ensured using the Armijo-Wolfe conditions, which also prevent excessively small step sizes. Convergence analysis is carried out under mild assumptions using the Tikhonov regularization framework. Numerical experiments, including examples involving neural networks, confirm the effectiveness of the proposed algorithm in reducing iteration count and achieving reliable solutions. 
The CFCG framework may be extended to quasi-Newton and Newton methods. Constrained optimization may be addressed, removing the current restriction to unconstrained problems. Adaptive strategies for selecting the fractional order and capturing memory effects will be investigated.
\section*{Declarations}
\begin{itemize}
    \item {\bf Funding informations:} Author Barsha Shaw appreciates the financial support received from the University Grants Commission (UGC) Fellowship (ID No. 211610020050), Government of India, which supported her Ph.D. work.
    \item {\bf Ethics approval and consent to participate:} Not applicable.
\item {\bf Consent for publication:} The authors declare no consent for publication.
\item {\bf Data availability statement:} The authors have used the fashion MNIST and KMNIST datasets collected from https://docs.pytorch.org/vision/stable/datasets.html.
\item {\bf Materials availability:} The simulation codes are available upon reasonable request from the corresponding author.
\item {\bf Competing interest:} The author declares no conflict of interest.
 \item {\bf Authors' contributions:} Both authors have equally contributed to the theoretical as well as numerical part.
\end{itemize}
 
\end{document}